\documentclass[12pt]{amsart}

\usepackage{amsfonts}
\usepackage{amsmath}
\usepackage{amssymb}
\usepackage{amsthm}
\usepackage{bm}
\usepackage{enumerate} 
\usepackage[margin=1.2in]{geometry}
\usepackage{tikz-cd}
\usepackage{hyperref}

\theoremstyle{plain}
\newtheorem{theorem}{Theorem}[section]
\newtheorem{lemma}[theorem]{Lemma}
\newtheorem{proposition}[theorem]{Proposition}

\theoremstyle{definition}

\newtheorem{remark}[theorem]{Remark}

\numberwithin{equation}{section}

\newcommand\N{\mathbb{N}}
\newcommand\Z{\mathbb{Z}}
\newcommand\R{\mathbb{R}}
\newcommand\C{\mathbb{C}}

\newcommand\D{\mathcal{D}}

\renewcommand\S{\mathcal{S}}

\newcommand\V{\mathcal{V}}
\newcommand\W{\mathcal{W}}

\newcommand\ev[2]{\langle#1,#2\rangle}
\newcommand\compltens{\widehat{\otimes}}

\DeclareMathOperator\im{im}
\DeclareMathOperator\id{id}
\DeclareMathOperator\csn{csn}

\DeclareMathOperator\condM{M}
\DeclareMathOperator\condN{N}
\DeclareMathOperator\condSq{Sq}

\begin{document}

\title[Kernel theorems for Beurling-Bj\"{o}rck spaces]{Kernel theorems for Beurling-Bj\"{o}rck type spaces}
\author[L. Neyt]{Lenny Neyt}
\address{Universit\"{a}t Trier\\ FB IV Mathematik\\ D-54286 Trier\\ Germany}
\thanks{L. Neyt was supported by  the Research Foundation--Flanders through the postdoctoral grant 12ZG921N and by a research fellowship of the Alexander von Humboldt Foundation}

\email{lenny.neyt@UGent.be}

\author[J. Vindas]{Jasson Vindas}
\thanks{J. Vindas gratefully acknowledges support by Ghent University through the BOF-grant 01J04017 and by the Research Foundation--Flanders through the FWO-grant G067621N}
\address{Department of Mathematics: Analysis, Logic and Discrete Mathematics\\ Ghent University\\ Krijgslaan 281\\ 9000 Gent\\ Belgium}
\email{jasson.vindas@UGent.be}


\begin{abstract}
We prove new kernel theorems for a general class of  Beurling-Bj\"{o}rck type spaces. In particular, our results cover the classical Beurling-Bj\"{o}rck spaces $\S^{(\omega)}_{(\eta)}$ and $\S^{\{\omega\}}_{\{\eta\}}$ defined via weight functions $\omega$ and $\eta$. 
\end{abstract}

\maketitle

\section{Introduction}

Schwartz type kernel theorems are essential tools in modern functional analysis as they allow us to represent continuous linear mappings via distribution kernels. Establishing kernel theorems or the related property of nuclearity for a given class of locally convex spaces are therefore questions of great relevance for the analysis of continuous linear mappings on such spaces. The study of these questions for locally convex spaces of functions defined through high time-frequency localization conditions has attracted much attention in recent times \cite{B-J-O-GaborWaveFrontSetUltradiffFunc,B-J-O-S21, C-N-2019, D-N-V-CharNuclBBSp, D-N-V-NuclGSSpKernThm, P-P-V-QuasiAnalClassGSParametrixConv,Teofanov2019,Toft2020}, particularly due to potential applications in the microlocal analysis of pseudodifferential and localization operators.

The aim of this article is to provide new kernel theorems for a general class of Beurling-Bj\"{o}rck type spaces. In order to motivate our results and outline the content of the paper, we state here the kernel theorems for the particular but important case of the classical Beurling-Bj\"{o}rck spaces \cite{B-LinPDOGenDist}. Let us introduce some notation. Let $\omega$ be a non-negative continuous function on $\R^{d}$. We consider the following standard and natural conditions \cite{B-LinPDOGenDist, B-M-T-UltraDiffFuncFourierAnal,D-N-V-CharNuclBBSp}:
	\begin{itemize}
		\item[$(\alpha)$] there are $L, C > 0$ such that $\omega(x + y) \leq L(\omega(x) + \omega(y)) + C$ for all $x, y \in \R^{d}$;
		\item[$(\gamma)$] $\log |x| = O( \omega(x))$ as $|x| \to \infty$;
		\item[$\{\gamma\}$] $\log |x| = o( \omega(x))$ as $|x| \to \infty$.
	\end{itemize}
Let $\eta$ be another non-negative continuous function on $\R^{d}$, then the classical Beurling-Bj\"{o}rck space (of Beurling type) $\S^{(\omega)}_{(\eta)}(\R^{d})$ consists of all those $\varphi \in \S^{\prime}(\R^{d})$ such that both $\varphi$ and $\widehat{\varphi}$ are continuous functions and
\begin{equation} 
\label{eq1BBdef} \operatorname*{sup}_{x \in \R^{d}} |\varphi(x)| e^{\lambda \omega(x)} + \operatorname*{sup}_{ \xi \in \R^{d}} |\widehat{\varphi}(\xi)| e^{\lambda \eta(\xi)} < \infty 
\end{equation}
for all $\lambda>0$, endowed with its natural Fr\'{e}chet space topology. Here $\widehat{\varphi}$ denotes the Fourier transform of $\varphi$. Likewise, the Beurling-Bj\"{o}rck space of Roumieu type $\S^{\{\omega\}}_{\{\eta\}}(\R^{d})$ is defined as the $(LB)$-space consisting of those $\varphi$ that satisfy \eqref{eq1BBdef} for some $\lambda>0$.  We will use $\S^{[\omega]}_{[\eta]}(\R^{d})$ as a common notation for $\S^{(\omega)}_{(\eta)}(\R^{d})$ and $\S^{\{\omega\}}_{\{\eta\}}(\R^{d})$, and a similar convention holds for other occurrences of the symbol $[\:]$.

We may now state our new kernel theorem for the classical Beurling-Bj\"{o}rck spaces. The notation $\omega_1 \oplus \omega_2$ stands for $(\omega_1 \oplus \omega_2)(x,y)=\omega_1(x)+\omega_2(y)$, $(x,y)\in \R^{d_1+d_2}$.

	\begin{theorem}
		\label{t:KernelTheoremsOriginalBBsp}
		Let $\omega_{j}$ and $\eta_{j}$ be weight functions on $\R^{d_{j}}$ satisfying $(\alpha)$ and $[\gamma]$ for $j = 1, 2$. The following canonical isomorphisms of locally convex spaces hold,
			\[ \S^{[\omega_{1} \oplus \omega_{2}]}_{[\eta_{1} \oplus \eta_{2}]}(\R^{d_{1} + d_{2}}) \cong \S^{[\omega_{1}]}_{[\eta_{1}]}(\R^{d_{1}}) \compltens \S^{[\omega_{2}]}_{[\eta_{2}]}(\R^{d_{2}}) \cong \mathcal{L}_{\beta}(\S^{[\omega_{1}]}_{[\eta_{1}]}(\R^{d_{1}})^{\prime}, \S^{[\omega_{2}]}_{[\eta_{2}]}(\R^{d_{2}})) \]
		and
			\[ \S^{[\omega_{1} \oplus \omega_{2}]}_{[\eta_{1} \oplus \eta_{2}]}(\R^{d_{1} + d_{2}})^{\prime} \cong \S^{[\omega_{1}]}_{[\eta_{1}]}(\R^{d_{1}})^{\prime} \compltens \S^{[\omega_{2}]}_{[\eta_{2}]}(\R^{d_{2}})^{\prime} \cong \mathcal{L}_{\beta}(\S^{[\omega_{1}]}_{[\eta_{1}]}(\R^{d_{1}}), \S^{[\omega_{2}]}_{[\eta_{2}]}(\R^{d_{2}})^{\prime}) . \]
	\end{theorem}
	
It should be noted that, as customary, the completed tensor products in Theorem \ref{t:KernelTheoremsOriginalBBsp} are with respect to either the $\varepsilon$-topology or the projective topology \cite{G-ProdTensEspNucl,P-NuclLCS,T-TopVecSp}, which is justified by the recent characterization of nuclearity for Beurling-Bj\"{o}rck spaces in terms of the conditions $(\alpha)$ and $[\gamma]$, see \cite[Theorem 1.1]{D-N-V-CharNuclBBSp}. 

 Theorem \ref{t:KernelTheoremsOriginalBBsp} is a direct improvement to  \cite[Theorem 7.3]{D-N-V-NuclGSSpKernThm}, where the kernel isomorphisms were attained under the more restrictive hypothesis that the $\omega_{j}$ are Braun-Meise-Taylor weight functions \cite{B-M-T-UltraDiffFuncFourierAnal}, that is, additionally assuming that they are radially increasing functions that satisfy the following condition from the theory of ultra\-differentiable functions (see e.g. \cite[p.~206]{B-M-T-UltraDiffFuncFourierAnal} or \cite[p.~211]{D-N-V-NuclGSSpKernThm}):
 	\begin{itemize}
		\item[$(\delta)$] $\phi : [0, \infty) \rightarrow [0, \infty)$, $\phi(x) = \omega(e^{x})$, is convex.
	\end{itemize}
That $(\delta)$ should play no role in obtaining kernel theorems is strongly suggested by the characterization of nuclearity from  \cite{D-N-V-CharNuclBBSp} quoted above. However, removing it from the set of hypotheses appears to require a new treatment of the problem. In fact, without $(\delta)$  one is not longer able to describe ultradifferentiability in terms of bounds on derivatives, which makes it unclear whether the standard way \cite{Schwartz57,K-UltraDist3} to prove kernel theorems applies in this situation. We shall overcome this difficulty with a new idea that combines mapping properties of the short-time Fourier transform with the representation of weighted spaces of continuous functions as $\varepsilon$-tensor products (cf. \cite{B-M-S-ProjDescrWeighIndLim}).

The plan of this article is as follows. In Section \ref{sec:WeighFuncSystem} we discuss properties of the families of weight function systems to be employed in this work. We introduce in Section \ref{sec:BBSp} natural generalizations of the Beurling-Bj\"{o}rck spaces and characterize when they are nuclear in Theorem \ref{t:NuclearityBBSp}. Finally, Section \ref{sec:KernelThms} is devoted to the new kernel theorems, our main result Theorem \ref{t:KernelTheorems} is a general form of Theorem \ref{t:KernelTheoremsOriginalBBsp}. To show the kernel theorems in the Roumieu case, we need to employ its so-called projective description, which we also establish in Theorem \ref{t:ProjDescr} using the short-time Fourier transform.

\section{Weight function systems}
\label{sec:WeighFuncSystem}

A \emph{weight function (on $\R^{d}$)} is a continuous function $\R^{d} \rightarrow [1, \infty)$. Set $\R_{+} = (0, \infty)$. Then, a \emph{weight function system (on $\R^{d}$)} is a family $\W = \{ w^{\lambda} : \lambda \in \R_{+} \}$ of weight functions on $\R^{d}$ such that $w^{\lambda}(x) \leq w^{\mu}(x)$ for all $x \in \R^{d}$ and $\mu \leq \lambda$. We consider the following conditions on a weight function system $\W$:
	\begin{itemize}
		\item[$(\condM)$] $\forall \lambda \in \R_{+} ~ \exists \mu, \nu \in \R_{+} ~ \exists C > 0 ~ \forall x, y \in \R^{d} ~ :~  w^\lambda(x+y)\leq Cw^{\mu}(x) w^{\nu}(y) $; 
		\item[$\{\condM\}$] $\forall  \mu, \nu \in \R_{+} ~ \exists \lambda \in \R_{+} ~ \exists C > 0 ~ \forall x, y \in \R^{d} ~ :~  w^\lambda(x+y)\leq Cw^{\mu}(x) w^{\nu}(y)$;
		\item[$(\condSq)$] $\forall \lambda, \mu \in \R_{+} ~ \exists \nu \in \R_{+} ~ \exists C > 0 ~ \forall x \in \R^{d} ~ : ~ w^{\lambda}(x) w^{\mu}(x) \leq C w^{\nu}(x)$;
		\item[$\{\condSq\}$] $\forall \nu \in \R_{+} ~ \exists \lambda, \mu \in \R_{+} ~ \exists C > 0 ~ \forall x \in \R^{d} ~ : ~ w^{\lambda}(x) w^{\mu}(x) \leq C w^{\nu}(x)$;
		\item[$(\condN)$] $\forall \lambda \in \R_{+}  ~ \exists \mu \in \R_{+} ~ :~  w^{\lambda} / w^{\mu} \in L^{1}$;
		\item[$\{\condN\}$] $\forall \mu \in \R_{+}  ~ \exists \lambda \in \R_{+} ~ :~  w^{\lambda} / w^{\mu} \in L^{1}$.
	\end{itemize}
We recall again that our convention is to employ  $[\:]$ as a common notation to treat both symbols $(\:)$ and $\{\}$ simultaneously. In addition, we will often first state assertions for $(\:)$ followed in parenthesis by the corresponding statement for $\{\:\}$. 
	
Given functions $w$ on $\R^{d_{1}}$ and $v$ on $\R^{d_{2}}$, we define $(w \otimes v)(x) = w(x_{1}) v(x_{2})$ for any $x = (x_{1}, x_{2}) \in \R^{d_{1} + d_{2}}$. Accordingly, for  weight function systems $\W$ on $\R^{d_{1}}$ and $\V$ on $\R^{d_{2}}$ we write $\W \otimes \V$ for the weight function system $\{ w^{\lambda} \otimes v^{\lambda} : \lambda \in \R_{+} \}$ on $\R^{d_{1} + d_{2}}$. Note that $\W \otimes \V$ satisfies the condition $[\condM]$, $[\condSq]$, or $[\condN]$ if and only if $\W$ and $\V$ do so. The function $\check{w}$ stands for the reflection about the origin of $w$, that is, $\check{w}(x)=w(-x)$; we also denote as $\check{\W}=\{\check{w}^{\lambda}:\lambda \in \R_{+}\}$ the corresponding weight function system obtained by reflecting $\W$.
 	
The following lemma will be needed later on. We denote by $C_{0}(\R^{d})$ the space of continuous functions on $\R^{d}$ vanishing at $\infty$.
	
	\begin{lemma}[{\cite[Lemma 3.1]{D-N-V-NuclGSSpKernThm}}]
		\label{l:CondNinC0}
		Let $\W$ be a weight function system satisfying $[\condM]$ and $[\condN]$. Then,
		\[ \forall \lambda \in \R_{+} ~ \exists \mu \in \R_{+} ~ ( \forall \mu \in \R_{+} ~ \exists \lambda \in \R_{+} ) ~ : ~ w^{\lambda} / w^{\mu} \in L^{1}(\R^{d}) \cap C_{0}(\R^{d}) . \]
	\end{lemma}
	
Let us now discuss how the weight functions considered in the introduction fit into the weight function system setting. Let $\omega$ be a non-negative continuous function on $\R^{d}$. We then associate to it the weight function system
	\[ \W_{\omega} = \{ e^{\frac{1}{\lambda} \omega} : \lambda \in \R_{+}  \} .  \]
We now have the following connections between the various conditions on $\omega$ and $\W_{\omega}$. We call a function $\omega$ radially increasing if $\omega(x) \leq \omega(y)$ whenever $|x| \leq |y|$.

	\begin{lemma}
		\label{l:WeighFuncSystemsSingleWeight}
		Let $\omega$ be a non-negative continuous function on $\R^{d}$. 
			\begin{itemize}
				\item[$(a)$] $\W_{\omega}$ satisfies $[\condM]$ if and only if $\omega$ satisfies $(\alpha)$.
				\item[$(b)$] $\W_{\omega}$ satisfies $[\condSq]$.
				\item[$(c)$] $\W_{\omega}$ satisfies $[\condN]$ if $\omega$ satisfies $[\gamma]$. If $\omega$ is radially increasing, then the converse implication also holds.
			\end{itemize}
	\end{lemma}
	
	\begin{proof}
		The verification of $(a)$ and $(b)$ is straightforward, so we omit it. To show $(c)$, we first note that $[\condN]$ holds for $\W_{\omega}$ if and only if for some $\varepsilon > 0$ (for every $\varepsilon > 0$) we have $e^{- \varepsilon \omega(\cdot)} \in L^{1}(\R^{d})$. It is clear that this is true when $\omega$ satisfies $[\gamma]$. Now, suppose that $\omega$ is radially increasing and that $\W_{\omega}$ satisfies $[\condN]$. Let $\varepsilon > 0$ be as above (let $\varepsilon > 0$ be arbitrary), then
			\[ |y|^{d} e^{-\varepsilon \omega(y)} \leq \frac{1}{|\overline{B}(0, 1)|} \int_{\overline{B}(0, |y|)} e^{-\varepsilon \omega(x)} dx \leq \frac{1}{|\overline{B}(0, 1)|} \int_{\R^{d}} e^{-\varepsilon \omega(x)} dx < \infty , \]
		for all $y \in \R^{d}$, where $\overline{B}(0, R) = \{ x \in \R^{d} : |x| \leq R \}$ for any $R \geq 0$. Hence we see that
			\[ \exists \varepsilon > 0 ~ (\forall \varepsilon > 0) ~ \exists C > 0 ~ \forall y \in \R^{d} : \quad d \log |y| \leq \varepsilon \omega(y) + C , \]
		whence we conclude that $[\gamma]$ holds.
	\end{proof}

Given a non-negative function $w$, we introduce the seminorm
	\[ \|\varphi\|_{w} = \operatorname*{sup}_{x \in \R^{d}} |\varphi(x)| w(x) . \]
Then, we denote by $C_{w} = C_{w}(\R^{d})$ the space of all continuous functions $\Phi$ on $\R^{d}$ such that $\|\Phi\|_{w} < \infty$. Naturally, if $w$ is positive and continuous, then $C_{w}$ is a Banach space. Given a weight function system $\W$, we define the spaces
	\[ C_{(\W)}= \varprojlim_{\lambda \rightarrow 0^{+}} C_{w^{\lambda}} , \qquad C_{\{\W\}} = \varinjlim_{\lambda \rightarrow \infty} C_{w^{\lambda}} . \]
Then $C_{(\W)}$ is a Fr\'{e}chet space and $C_{\{\W\}}$ is an $(LB)$-space. Furthermore, $C_{\{\W\}}$ is always complete \cite{B-B91}.

In the Roumieu case we will consider the so-called projective description of $C_{\{\W\}}$. The \emph{maximal Nachbin family associated to $\W$}, denoted by $\overline{V}(\W)$, is given by the space consisting of all non-negative upper semicontinuous functions $w$ on $\R^{d}$ such that $\sup_{x \in \R^{d}} w(x) / w^{\lambda}(x) < \infty$ for all $\lambda \in \R_{+}$. The \emph{projective hull of $C_{\{\W\}}$} is defined as the space $C\overline{V}(\W)$ of all continuous functions $\varphi$ on $\R^{d}$ such that $\|\varphi\|_{w} < \infty$ for all $w \in \overline{V}(\W)$. We endow $C\overline{V}(\W)$ with the locally convex topology generated by the system of seminorms $\{ \|\cdot\|_{w} : w \in \overline{V}(\W) \}$. The spaces $C_{\{\W\}}$ and $C\overline{V}(\W)$ are always equal as sets, which follows directly from the next lemma.
	
	\begin{lemma}
		\label{l:BoundednessNachbinFamily}
		Let $\W$ be a weight function system. For any set $B$ of functions on $\R^{d}$ we have that $\sup_{\varphi \in B} \sup_{x \in \R^{d}} |\varphi(x)| w^{\lambda}(x) < \infty$ for some $\lambda \in \R_{+}$ if and only if $\sup_{\varphi \in B} \sup_{x \in \R^{d}} |\varphi(x)| w(x) < \infty$ for every $w \in \overline{V}(\W)$.
	\end{lemma}
	
	\begin{proof}
		This is shown similarly as in \cite[Lemma 4.11]{D-V-WeightedIndUltra}.
	\end{proof}
	
They coincide topologically in the following case.

	\begin{lemma}[{\cite[Corollary 5, p.~116]{B-IntroLocallyConvexIndLim}}]
		\label{l:ProjDescrContinuousFunctions}
		If $\W$ is a weight function system satisfying
			\begin{itemize}
				\item[$(S)$] $\forall \mu \in \R_{+} ~ \exists \lambda \in \R_{+} : \: w^{\lambda} / w^{\mu} \in C_{0}(\R^{d})$, 
			\end{itemize}
		then,  $C_{\{\W\}} = C\overline{V}(\W)$ as locally convex spaces.
	\end{lemma}
	
Note that due to Lemma \ref{l:CondNinC0}, if $\W$ satisfies $\{\condM\}$ and $\{\condN\}$, then it also satisfies $(S)$. If $\W_{j}$ is a weight function system on $\R^{d_{j}}$, $j = 1, 2$, then $\W_{1} \otimes \W_{2}$ satisfies $(S)$ if and only if the $\W_{j}$ do so. Moreover, $\overline{V}(\W_{1}) \otimes \overline{V}(\W_{2})$ is upward dense in $\overline{V}(\W_{1} \otimes \W_{2})$, that is, for every $w \in \overline{V}(\W_{1} \otimes \W_{2})$ there exist $w_{j} \in \overline{V}(\W_{j})$, $j = 1, 2$, such that $w(x_{1}, x_{2}) \leq w_{1} \otimes w_{2}(x_{1}, x_{2})$ for all $(x_{1}, x_{2}) \in \R^{d_{1} + d_{2}}$. 
	
We conclude this section with a remark about condition $[\condSq]$.
	
	\begin{remark}
		The condition $[\condSq]$ was first introduced in \cite{D-N-WeighPLBSpUltraDiffMultSp} and is a natural property in the context of Beurling-Bj\"{o}rck spaces as seen from Lemma \ref{l:WeighFuncSystemsSingleWeight}. It also naturally arises in the context of the so-called Gelfand-Shilov spaces. Indeed, suppose $M = (M_{p})_{p \in \N}$ is a sequence of positive numbers satisfying the condition $(M.1)$ from \cite{K-Ultradistributions1} and is such that $\lim_{p \rightarrow \infty} M_{p}^{1/p} = \infty$. Then, consider its \emph{associated function}
			\[ \omega_{M}(x) = \sup_{p \in \N}\log \frac{|x|^{p} M_{0}}{M_{p}} , \qquad x \in \R^{d} \setminus \{0\} , \]
		and $\omega_{M}(0) = 1$. Now, if $A = (A_{p})_{p \in \N}$ is another sequence of positive numbers satisfying $(M.1)$ and $\lim_{p \rightarrow \infty} A_{p}^{1/p} = \infty$ with associated function $\omega_{A}$, the Gelfand-Shilov space (of Beurling and Roumieu type) $\S^{[M]}_{[A]}$ (see \cite{G-S-GenFunc2}) consists of all those $\varphi \in C^{\infty}(\R^{d})$ such that
			\[ \forall \ell > 0 ~ (\exists \ell > 0) : \quad \sup_{(\alpha, x) \in \N^{d} \times \R^{d}} \frac{|\varphi^{(\alpha)}(x)| e^{\omega_{A}(x / \ell)}}{\ell^{|\alpha|} M_{|\alpha|}} < \infty , \]
		endowed with its natural Fr\'{e}chet space topology ($(LB)$-space topology). We can associate the weight function system $\W_{M} = \{ e^{\omega_{M}(\cdot / \lambda)} : \lambda \in \R_{+} \}$ to the weight sequence $M$, and similarly we associate the weight function system $\W_{A}$ to $A$. If we now alternatively consider the Beurling-Bj\"{o}rck space $\S^{[\W_{M}]}_{[\W_{A}]}$ (see Section \ref{sec:BBSp}), then one can show (cf. \cite[Corollary 2.4]{C-C-K-CharGSSpFourierTrans}) that if $M$ and $A$ satisfy Komatsu's condition $(M.2)$, then $\S^{[M]}_{[A]}$ and $\S^{[\W_{M}]}_{[\W_{A}]}$ coincide topologically. Now $M$, respectively $A$, satisfies $(M.2)$ if and only if $\W_{M}$, respectively $\W_{A}$, satisfies $[\condSq]$ \cite[Lemma 3.6]{D-N-WeighPLBSpUltraDiffMultSp}. It is worthwhile however to point out that for the Gelfand-Shilov spaces analogous results to those in this paper were obtained in \cite{D-N-V-NuclGSSpKernThm} under the weaker condition $(M.2)'$ on $M$ and $A$, which turns out to be equivalent to $\W_{M}$ and $\W_{A}$ satisfying $[\condN]$ \cite[Lemma 3.4]{D-N-V-NuclGSSpKernThm}.
	\end{remark}

\section{The Beurling-Bj\"orck spaces}
\label{sec:BBSp}

Let $v$ and $w$ be non-negative functions on $\R^{d}$, then we denote by $\S^{v}_{w} = \S^{v}_{w}(\R^{d})$ the seminormed space of all $\varphi \in \S^{\prime}(\R^{d})$ such that both $\varphi$ and $\widehat{\varphi}$ are continuous and $\|\varphi\|_{\S^{v}_{w}} = \|\varphi\|_{w} + \|\widehat{\varphi}\|_{v} < \infty$, where $\widehat{\varphi}$ stands for the Fourier transform whose constants we fix as
	\[ \mathcal{F}(\varphi)(\xi) = \widehat{\varphi}(\xi) = \int_{\R^{d}} \varphi(t) e^{-2 \pi i \xi \cdot t} d\xi . \]
If $v$ and $w$ are positive and continuous, then $\S^{v}_{w}$ is a Banach space. Given two weight function systems $\V$ and $\W$, we define the \emph{Beurling-Bj\"{o}rck space (of Beurling and Roumieu type)}
	\[ \S^{(\V)}_{(\W)} = \varprojlim_{\lambda \rightarrow 0^{+}} \S^{v^{\lambda}}_{w^{\lambda}} , \qquad \S^{\{\V\}}_{\{\W\}} = \varinjlim_{\lambda \rightarrow \infty} \S^{v^{\lambda}}_{w^{\lambda}} . \]
Note that $\S^{(\V)}_{(\W)}$ is a Fr\'echet space, while $\S^{\{\V\}}_{\{\W\}}$ is an $(LB)$-space. We remark that the Fourier transform is an isomorphism between $\S^{[\V]}_{[\W]}$ and $\S_{[\V]}^{[\check{\W}]}$. If $\W$ satisfies $[\condM]$, then $\S^{[\V]}_{[\W]}$ is translation-invariant; a fact we shall tacitly use in the sequel. In the Roumieu case, we additionally have the following result on the topology of $\S^{\{\V\}}_{\{\W\}}$.

	\begin{proposition}
		\label{p: completeness inductive} 
		Let $\V$ and $\W$ be weight function systems. Then, $\S^{\{\V\}}_{\{\W\}}$ is a complete and thus regular $(LB)$-space.
	\end{proposition}
	
	\begin{proof} 
		A basis of neighborhoods of the origin in  $\S^{\{\V\}}_{\{\W\}}$ is obtained by taking sets of the form $W \cap \mathcal{F}^{-1}(V \cap \S^{\prime}(\R^{d}))$ where $W$ is a neighborhood of the origin in $C_{\{\W\}}$ and $V$ in $C_{\{\V\}}$. The completeness of $\S^{\{\V\}}_{\{\W\}}$ then follows from that of $C_{\{\W\}}$ and $C_{\{\V\}}$.
	\end{proof}

The primary goal of this section will be to characterize when the Beurling-Bj\"orck spaces are nuclear. In fact, we have the following result.

	\begin{theorem}
		\label{t:NuclearityBBSp}
		Let $\V$ and $\W$ be two weight function systems satisfying $[\condM]$ and $[\condSq]$ for which $\S^{[\V]}_{[\W]} \neq \{0\}$. Then, the following are equivalent:
			\begin{itemize}
				\item[$(i)$] $\V$ and $\W$ satisfy $[\condN]$.
				\item[$(ii)$] $\S^{[\V]}_{[\W]}$ is nuclear.
			\end{itemize}
	\end{theorem}
	
We only show the implication $(i) \Rightarrow (ii)$ in this section, the proof of $(ii) \Rightarrow (i)$ will be postponed until Appendix \ref{appendix:ProofNecessityNuclearity}. We start by introducing an integrable variant of the Beurling-Bj\"{o}rck spaces. Let $w$ be a non-negative measurable function on $\R^{d}$, then we consider the seminorm
	\[ \|\varphi\|_{w, 1} = \int_{\R^{d}} |\varphi(x)| w(x) dx , \]
and for another non-negative measurable function $v$ on $\R^{d}$ we define $\S^{v}_{w, 1} = \S^{v}_{w, 1}(\R^{d})$ as the seminormed space of all those $\varphi \in \S^{\prime}(\R^{d})$ such that $\|\varphi\|_{\S^{v}_{w, 1}} = \|\varphi\|_{w, 1} + \|\widehat{\varphi}\|_{v, 1} < \infty$. For two weight function systems $\V$ and $\W$ on $\R^{d}$ we then define the spaces
	\[ \S^{(\V)}_{(\W), 1} = \varprojlim_{\lambda \rightarrow 0^{+}} \S^{v^{\lambda}}_{w^{\lambda}, 1} , \qquad \S^{\{\V\}}_{\{\W\}, 1} = \varinjlim_{\lambda \rightarrow \infty} \S^{v^{\lambda}}_{w^{\lambda}, 1} . \]
Note that if $v$ is positively bounded from below, then the functions in $\S^{v}_{w, 1}$ are continuous. In particular, the elements of $\S^{[\V]}_{[\W], 1}$ are all continuous functions. Our first goal is to show that if the weight function systems satisfy the conditions $[\condM]$, $[\condSq]$, and $[\condN]$, then $\S^{[\V]}_{[\W]} = \S^{[\V]}_{[\W], 1}$ as locally convex spaces. One inclusion is obvious, and its proof is left to the reader.

	\begin{lemma}
		\label{l:BBSpInftyInBBSpL1}
		Suppose $\V$ and $\W$ are weight function systems satisfying $[\condN]$. Then, $\S^{[\V]}_{[\W]} \subseteq \S^{[\V]}_{[\W], 1}$ continuously. 
	\end{lemma}
	
We now study when $\S^{[\V]}_{[\W], 1} \subseteq \S^{[\V]}_{[\W]}$. We first observe that the non-triviality of $\S^{[\V]}_{[\W], 1}$ implies that of $\S^{[\V]}_{[\W]}$.

	\begin{lemma}
		\label{l:BBSpNonTrivialEquiv}
		Let $\V$ and $\W$ be weight function systems satisfying $[\condM]$. If $\S^{[\V]}_{[\W]} \neq \{0\}$ or $\S^{[\V]}_{[\W], 1} \neq \{0\}$, then, $\S^{[\V]}_{[\W]} \cap \S^{[\V]}_{[\W], 1} \neq \{0\}$.
	\end{lemma}
	
	\begin{proof}
		Let $\varphi$ be a non-zero element of the corresponding non-trivial space. Pick $\psi, \chi \in \D(\R^{d})$ for which $(\varphi * \psi)(0)=\langle \varphi, \check{\psi} \rangle= 1$ and $\widehat{\chi}(0) = 1$. Then one readily verifies that $\varphi_{0} = (\varphi * \psi) \cdot \widehat{\chi}$ is an element of both $\S^{[\V]}_{[\W]}$ and  $\S^{[\V]}_{[\W], 1}$ that is non-trivial (as $\varphi_{0}(0) = 1$).
	\end{proof}

 Our proof of $\S^{[\V]}_{[\W], 1} \subseteq \S^{[\V]}_{[\W]}$ is based on the mapping properties of the \emph{short-time Fourier transform} (STFT). The STFT of a function $f \in L^{2}(\R^{d})$ with respect to a window $\psi \in L^{2}(\R^{d})$ is given by
	\[ V_{\psi} f(x, \xi) = \int_{\R^{d}} f(t) \overline{\psi(t - x)} e^{-2\pi i \xi \cdot t} dt , \qquad (x, \xi) \in \R^{2d} . \]
Then $V_{\psi}$ is a continuous linear mapping $L^{2}(\R^{d}) \rightarrow L^{2}(\R^{2d})$. The \emph{adjoint STFT} of a function $F \in L^{2}(\R^{2d})$ is given by the weak integral
	\[ V^{*}_{\psi} F(t) = \iint_{\R^{2d}} F(x, \xi) e^{2 \pi i \xi \cdot t} \psi(t - x) dx d\xi . \]
A function $\gamma \in L^{2}(\R^{d})$ is called a \emph{synthesis window} for $\psi$ if $(\gamma, \psi)_{L^{2}} \neq 0$, and in this case we have
	\begin{equation}
		\label{eq:STFTReconstruct}
		\frac{1}{(\gamma, \psi)_{L^{2}}} V^{*}_{\gamma} \circ V_{\psi} = \id_{L^{2}(\R^{d})} . 
	\end{equation}
We now consider the (adjoint) STFT in the context of our general Beurling-Bj\"{o}rck spaces.
	
	\begin{lemma}
		\label{l:STFTSpecific}
		Let $v_{j}$ and $w_{j}$ be non-negative measurable functions on $\R^{d}$ for $j \in \{0,1,2,3\}$. Suppose that for certain $C_{0}, C_{1} > 0$ we have
			\begin{equation}
				\label{eq:STFTSpecificCondSq}
				v_{0}(x)^{2} \leq C_{0} v_{1}(x) \quad \text{and} \quad w_{0}(x)^{2} \leq C_{0} w_{1}(x) , \qquad \forall x \in \R^{d} ,
			\end{equation}
		and
			\begin{equation}
				\label{eq:STFTSpecificCondM}
				v_{1}(x + y) \leq C_{1} v_{2}(x) v_{3}(y) \quad \text{and} \quad w_{1}(x + y) \leq C_{1} w_{2}(x) w_{3}(y) , \qquad \forall x, y \in \R^{d} .
			\end{equation}
		Then, for any $\psi \in \S^{v_{3}}_{w_{3}}$, the linear mapping
			\[ V_{\check{\psi}} : \S^{v_{2}}_{w_{2}, 1} \rightarrow C_{w_{0} \otimes v_{0}}(\R^{2d}) \]
		is well-defined and continuous.
	\end{lemma}
	
	\begin{proof}
		Take any $\varphi \in \S^{v_{2}}_{w_{2}, 1}$. On the one hand, we have
			\begin{align*}
				\sup_{(x, \xi) \in \R^{2d}} |V_{\check{\psi}} \varphi(x, \xi)| w_{1}(x) 
					&
					\leq C_{1} \int_{\R^{d}} |\varphi(t)| w_{2}(t) |\psi(x - t)| w_{3}(x - t) dt 
					\\
					&
					\leq C_{1} \|\psi\|_{w_{3}} \|\varphi\|_{w_{2}, 1} ,
			\end{align*}
		while on the other hand,
			\begin{multline*}
				\sup_{(x, \xi) \in \R^{2d}} |V_{\check{\psi}} \varphi(x, \xi)| v_{1}(\xi)
					= \sup_{(x, \xi) \in \R^{2d}} |V_{\mathcal{F}(\check{\varphi})} \widehat{\psi}(\xi, -x)| v_{1}(\xi) \\
					\leq C_{1} \int_{\R^{d}} |\widehat{\varphi}(\xi - t)| v_{2}(\xi - t) |\widehat{\psi}(t)| v_{3}(t) dt
					\leq C_{1} \|\widehat{\psi}\|_{v_{3}} \|\widehat{\varphi}\|_{v_{2}, 1} .
			\end{multline*}
		Setting $C = C_{0} C_{1} \|\psi\|_{\S^{v_{3}}_{w_{3}}}$, we get
			\begin{align*} 
				\sup_{(x, \xi) \in \R^{2d}} |V_{\check{\psi}} \varphi(x, \xi)| w_{0}(x) v_{0}(\xi) 
					&\leq \sup_{(x, \xi) \in \R^{2d}} |V_{\check{\psi}} \varphi(x, \xi)| (\max\{w_{0}(x), v_{0}(\xi)\})^{2} \\
					&\leq C_{0} \sup_{(x, \xi) \in \R^{2d}} |V_{\check{\psi}} \varphi(x, \xi)| (w_{1}(x) + v_{1}(\xi)) \\
					&\leq C \|\varphi\|_{\S^{v_{2}}_{w_{2}, 1}} , 
				\end{align*}
		which shows the result.
	\end{proof}
	
	\begin{lemma}
		\label{l:AdjointSTFTSpecific}
		Let $v_{j}$ and $w_{j}$ be non-negative measurable functions on $\R^{d}$ for $j \in \{0,1,2,3\}$. Suppose that \eqref{eq:STFTSpecificCondM} holds, that $\inf_{x \in \R^{d}} v_{2}(x) > 0$ and $\inf_{x \in \R^{d}} w_{2}(x) > 0$, and
			\begin{equation}
				\label{eq:AdjointSTFTSpecificCondN}
				v_{2} / v_{0} \in L^{1}(\R^{d}) \quad \text{and} \quad w_{2} / w_{0} \in L^{1}(\R^{d}) .
			\end{equation}
		Then, for any $\psi \in \S^{v_{3}}_{w_{3}}$, the linear mapping
			\[ V^{*}_{\psi} : C_{w_{0} \otimes v_{0}}(\R^{2d}) \rightarrow \S^{v_{1}}_{w_{1}} \]
		is well-defined and continuous.
	\end{lemma}
	
	\begin{proof}
		Take any $\Phi \in C_{w_{0} \otimes v_{0}}(\R^{2d})$. First, set $\varepsilon_{1} = \inf_{x \in \R^{d}} v_{2}(x)$ and note that $\varepsilon_{1} / v_{0} \leq v_{2} / v_{0} \in L^{1}(\R^{d})$, so that
			\begin{align*}
				\sup_{t \in \R^{d}} |V^{*}_{\psi}\Phi(t)| w_{1}(t) 
					&\leq C_{1} \sup_{t \in \R^{d}}  \iint_{\R^{2d}} |\Phi(x, \xi)| w_{2}(x) |\psi(t - x)| w_{3}(t - x) dxd\xi \\
					&\leq \varepsilon_{1}^{-1} C_{1} \|v_{2} / v_{0}\|_{L^{1}} \|w_{2} / w_{0}\|_{L^{1}} \|\psi\|_{w_{3}} \|\Phi\|_{w_{0} \otimes v_{0}} .
			\end{align*}
		We now set $\varepsilon_{2} = \inf_{x \in \R^{d}} w_{2}(x)$ and note that $\varepsilon_{2} / w_{0} \leq w_{2} / w_{0} \in L^{1}(\R^{d})$, hence
			\begin{align*}
				\sup_{t \in \R^{d}} |\mathcal{F}(V^{*}_{\psi}\Phi)(t)| v_{1}(t)
					&= \sup_{t \in \R^{d}} \left|\iint_{\R^{2d}} \Phi(x, \xi) \widehat{\psi}(t - \xi) e^{-2\pi i t \cdot x} e^{2 \pi i \xi \cdot x} dxd\xi\right| v_{1}(t) \\
					&\leq C_{1} \sup_{t \in \R^{d}} \iint_{\R^{2d}} |\Phi(x, \xi)| v_{2}(\xi) |\widehat{\psi}(t - \xi)| v_{3}(t - \xi) dxd\xi \\
					&\leq \varepsilon_{2}^{-1} C_{1} \|v_{2} / v_{0}\|_{L^{1}} \|w_{2} / w_{0}\|_{L^{1}} \|\widehat{\psi}\|_{v_{3}} \|\Phi\|_{w_{0} \otimes v_{0}} .
			\end{align*}
		Putting $C = (\varepsilon_{1}^{-1} + \varepsilon_{2}^{-1}) C_{1} \|v_{2} / v_{0}\|_{L^{1}} \|w_{2} / w_{0}\|_{L^{1}} \|\psi\|_{\S^{v_{3}}_{w_{3}}}$, we get
			\[ \|V^{*}_{\psi}\Phi\|_{\S^{v_{1}}_{w_{1}}} \leq C \|\Phi\|_{w_{0} \otimes v_{0}} , \]
		whence the mapping is well-defined and continuous.
	\end{proof}

	\begin{proposition}
		\label{p:STFTBBSp}
		Let $\V$ and $\W$ be weight function systems satisfying $[\condM]$, $[\condSq]$, and $[\condN]$. For any $\psi \in \S^{[\V]}_{[\W]}$,
		the linear mappings
			\[ V_{\check{\psi}} : \S^{[\V]}_{[\W], 1} \rightarrow C_{[\W \otimes \V]}(\R^{2d}) \quad \text{and} \quad V^{*}_{\psi} : C_{[\W \otimes \V]}(\R^{2d}) \rightarrow \S^{[\V]}_{[\W]}  \]
		are well-defined and continuous. Moreover, if $\gamma \in \S^{[\V]}_{[\W]}$ is a synthesis window for $\check{\psi}$, then,
			\begin{equation}
				\label{eq:STFTReconstructBBSp}
				\frac{1}{(\gamma, \check{\psi})_{L^{2}}} V^{*}_{\gamma} \circ V_{\check{\psi}} = \id_{\S^{[\V]}_{[\W], 1}} . 
			\end{equation}
	\end{proposition}
	
	\begin{proof}
		Lemmas \ref{l:STFTSpecific} and \ref{l:AdjointSTFTSpecific} imply that the linear mappings are well-defined and continuous. Now, take any $\varphi \in \S^{[\V]}_{[\W], 1}$, then due to the condition $[\condN]$, the functions $\varphi(\cdot) \overline{\psi}(x - \cdot)$ and $V_{\check{\psi}} \varphi(x, \cdot)$, for fixed $x \in \R^{d}$, both belong to $L^{1}(\R^{d})$. As $V_{\check{\psi}} \varphi(x, \xi) = \mathcal{F}(\varphi(\cdot) \overline{\psi}(x - \cdot))(\xi)$, we obtain that
			\begin{align*}
				\iint_{\R^{2d}} V_{\check{\psi}}\varphi(x, \xi) \gamma(t - x) e^{2 \pi i \xi \cdot t} dx d\xi
					&= \int_{\R^{d}} \left( \int_{\R^{d}} V_{\check{\psi}} \varphi(x, \xi) e^{2 \pi i \xi \cdot t} d\xi \right) \gamma(t - x) dx \\
					&= \varphi(t) \int_{\R^{d}} \overline{\psi}(x - t) \gamma(t - x) dx = (\gamma, \check{\psi})_{L^{2}} \varphi(t)
			\end{align*}
		for all $t \in \R^{d}$.
	\end{proof}
	
	\begin{proposition}
		\label{p:BBSpInftyandL1}
		Let $\V$ and $\W$ be weight function systems satisfying $[\condM]$, $[\condSq]$, and $[\condN]$. Then, $\S^{[\V]}_{[\W]} = \S^{[\V]}_{[\W], 1}$ as locally convex spaces.
	\end{proposition}
	
	\begin{proof}
		By Lemma \ref{l:BBSpInftyInBBSpL1} we only have to show that $\S^{[\V]}_{[\W], 1} \subseteq \S^{[\V]}_{[\W]}$ continuously and we may assume that $\S^{[\V]}_{[\W], 1} \neq \{0\}$. Lemma \ref{l:BBSpNonTrivialEquiv} yields $\S^{[\V]}_{[\W]} \neq \{0\}$. We then pick a non-trivial $\psi \in \S^{[\V]}_{[\W]}$. Since $\check{\psi}\ast \check{\overline{\psi}}\neq0$, we can select some $x\in\R^d$ such that $\gamma = \psi(\:\cdot- x\:) \in \S^{[\V]}_{[\W]}$ is a synthesis window for $\check{\psi}$. We may assume that $(\gamma,\check{\psi})_{L^{2}} = 1$. Proposition \ref{p:STFTBBSp} then says that $\id_{\S^{[\V]}_{[\W], 1}} = V^{*}_{\gamma} \circ V_{\check{\psi}} : \S^{[\V]}_{[\W], 1} \rightarrow \S^{[\V]}_{[\W]}$ is continuous, which yields the assertion.
	\end{proof}
		
We now show the implication $(i) \Rightarrow (ii)$ of Theorem \ref{t:NuclearityBBSp} by means of Grothendieck's criterion for nuclearity in terms of summable sequences \cite{G-ProdTensEspNucl}. Let $E$ be a lcHs (= Hausdorff locally convex space) and denote by $\csn(E)$ the set of all continuous seminorms on $E$. We call a sequence $(e_{n})_{n \in \N}$ in $E$ \emph{weakly summable} if $\sum_{n \in \N} |\ev{e^{\prime}}{e_{n}}| < \infty$ for any $e^{\prime} \in E^{\prime}$. As a consequence of Mackey's theorem, a sequence $(e_{n})_{n \in \N}$ in $E$ is weakly summable if and only if the set
	\[ \bigcup_{k \in \N} \{ \sum_{n = 0}^{k} c_{n} e_{n} : |c_{n}| \leq 1, n = 0, \ldots, k \} \]
is bounded in $E$. A sequence $(e_{n})_{n \in \N}$ in $E$ is called \emph{absolutely summable} if $\sum_{n \in \N} p(e_{n}) < \infty$ for any $p \in \csn(E)$ (or equivalently, for any $p$ in a fundamental system of seminorms on $E$). Clearly, any absolutely summable sequence is weakly summable. In case of a Fr\'{e}chet space or a $(DF)$-space, the validity of the converse implication is equivalent to nuclearity. 

	\begin{proposition}[{\cite[Theorem 4.2.5]{P-NuclLCS}}]
		\label{p:NuclEquivSequences}
		Let $E$ be a Fr\'{e}chet space or a $(DF)$-space. Then, $E$ is nuclear if and only if every weakly summable sequence in $E$ is absolutely summable.
	\end{proposition}
	
	\begin{proof}[Proof of Theorem \ref{t:NuclearityBBSp} $(i) \Rightarrow (ii)$]
		We will show the nuclearity of $\S^{[\V]}_{[\W]}$ using Proposition \ref{p:NuclEquivSequences}. For it, let $(\varphi_{n})_{n \in \N}$ be a weakly summable sequence in $\S^{[\V]}_{[\W]}$. Then, for any $\lambda \in \R_{+}$ (for some $\lambda \in \R_{+}$ by Proposition \ref{p: completeness inductive}) there is a $C > 0$ such that
			\[ \Big\|\sum_{n = 0}^{k} c_{n} \varphi_{n}\Big\|_{\S^{v^{\lambda}}_{w^{\lambda}}} \leq C \]
		for all $k \in \N$ and $|c_{n}| \leq 1$, $n = 0, \ldots, k$. We first show that
			\begin{equation}
				\label{eq:NuclProofEq}
				\sup_{x \in \R^{d}} \sum_{n \in \N} |\varphi_{n}(x)| w^{\lambda}(x) \leq C \quad \text{and} \quad \sup_{\xi \in \R^{d}} \sum_{n \in \N} |\widehat{\varphi}_{n}(\xi)| v^{\lambda}(\xi) \leq C .
			\end{equation} 
		Fix $x, \xi \in \R^{d}$ and choose $|c^{1}_{n}(x)| \leq 1$ and $|c^{2}_{n}(\xi)| \leq 1$ such that $c^{1}_{n}(x) \varphi_{n}(x) = |\varphi_{n}(x)|$ and $c^{2}_{n}(\xi) \widehat{\varphi}_{n}(\xi) = |\widehat{\varphi}_{n}(\xi)|$ for any $n \in \N$. Then, for arbitrary $k \in \N$, we have
			\[ \sum_{n = 0}^{k} |\varphi_{n}(x)| w^{\lambda}(x) = \left|\sum_{n = 0}^{k} c^{1}_{n}(x) \varphi_{n}(x)\right| w^{\lambda}(x) \leq C , \]
		and 
			\[ \sum_{n = 0}^{k} |\widehat{\varphi}_{n}(\xi)| v^{\lambda}(\xi) = \left|\sum_{n = 0}^{k} c^{2}_{n}(\xi) \widehat{\varphi}_{n}(\xi)\right| v^{\lambda}(\xi) \leq C , \]
		whence \eqref{eq:NuclProofEq} follows by letting $k \rightarrow \infty$. Because of Proposition \ref{p:BBSpInftyandL1}, to prove the nuclearity it suffices to show that
			\[ \sum_{n = 0}^{\infty} \|\varphi_{n}\|_{\S^{v^{\mu}}_{w^{\mu}, 1}} < \infty \]
		for all $\mu \in \R_{+}$ (for some $\mu \in \R_{+}$). Let $\mu \in \R_{+}$ be arbitrary (let $\lambda \in \R_{+}$ be such that \eqref{eq:NuclProofEq} holds), then $[\condN]$ implies there is some $\lambda \in \R_{+}$ (some $\mu \in \R_{+}$) such that $w^{\mu} / w^{\lambda} \in L^{1}$ and $v^{\mu} / v^{\lambda} \in L^{1}$. We get
			\begin{align*} 
				\sum_{n = 0}^{\infty} \|\varphi_{n}\|_{\S^{v^{\mu}}_{w^{\mu}, 1}} 
					&= \sum_{n = 0}^{\infty} \int_{\R^{d}} |\varphi_{n}(x)| w^{\mu}(x) dx + \int_{\R^{d}} |\widehat{\varphi}_{n}(\xi)| v^{\mu}(\xi) d\xi \\
					&= \int_{\R^{d}} \sum_{n = 0}^{\infty} |\varphi_{n}(x)| w^{\lambda}(x) \frac{w^{\mu}(x)}{w^{\lambda}(x)} dx + \int_{\R^{d}} \sum_{n = 0}^{\infty} |\widehat{\varphi}_{n}(\xi)| v^{\lambda}(\xi) \frac{v^{\mu}(\xi)}{v^{\lambda}(\xi)} d\xi \\
					&\leq C(\|w^{\mu} / w^{\lambda}\|_{L^{1}} + \|v^{\mu} / v^{\lambda}\|_{L^{1}}) .
			\end{align*}
	\end{proof}

\section{The kernel theorems}
\label{sec:KernelThms}

In this section we show the main result of this paper, that is, the kernel theorems for the spaces $\S^{[\V]}_{[\W]}$.

	\begin{theorem}
		\label{t:KernelTheorems}
		Let $\V_{j}$ and $\W_{j}$ be weight function systems on $\R^{d_{j}}$ satisfying $[\condM]$, $[\condSq]$, and $[\condN]$ for $j = 1, 2$. The following canonical isomorphisms of locally convex spaces hold
			\begin{equation}
				\label{eq:KernelTheoremBBSp}
				\S^{[\V_{1} \otimes \V_{2}]}_{[\W_{1} \otimes \W_{2}]}(\R^{d_{1} + d_{2}}) \cong \S^{[\V_{1}]}_{[\W_{1}]}(\R^{d_{1}}) \compltens \S^{[\V_{2}]}_{[\W_{2}]}(\R^{d_{2}}) \cong \mathcal{L}_{\beta}(\S^{[\V_{1}]}_{[\W_{1}]}(\R^{d_{1}})^{\prime}, \S^{[\V_{2}]}_{[\W_{2}]}(\R^{d_{2}}))
			\end{equation}
		and
			\begin{equation}
				\label{eq:KernelTheoremDual}
				\S^{[\V_{1} \otimes \V_{2}]}_{[\W_{1} \otimes \W_{2}]}(\R^{d_{1} + d_{2}})^{\prime} \cong \S^{[\V_{1}]}_{[\W_{1}]}(\R^{d_{1}})^{\prime}\compltens \S^{[\V_{2}]}_{[\W_{2}]}(\R^{d_{2}})^{\prime} \cong \mathcal{L}_{\beta}(\S^{[\V_{1}]}_{[\W_{1}]}(\R^{d_{1}}), \S^{[\V_{2}]}_{[\W_{2}]}(\R^{d_{2}})^{\prime}) .
			\end{equation}
	\end{theorem}

Let us first briefly recall some standard notions about tensor products for the sake of the reader's convenience. We always topologize dual spaces with the strong dual topology. Given two lcHs $E$ and $F$, we write $\mathcal{L}_{\beta}(E, F)$ for the space of continuous linear mappings from $E$ into $F$ endowed with the topology of uniform convergence over the bounded subsets of $E$. For the tensor product $E \otimes F$, we write $E \otimes_{\varepsilon} F$, respectively $E \otimes_{\pi} F$, if we put on it the $\varepsilon$-topology, respectively the projective topology. We write $E \widehat{\otimes}_{\varepsilon} F$ and $E \widehat{\otimes}_{\pi} F$ for the completion of the tensor product $E \otimes F$ with respect to the $\varepsilon$-topology and the projective topology, respectively. If either $E$ or $F$ is nuclear, then $E \widehat{\otimes}_{\varepsilon} F = E \widehat{\otimes}_{\pi} F$ and we drop the subscripts $\varepsilon$ and $\pi$ in the notation.

To show Theorem \ref{t:KernelTheorems} in the Roumieu case, we will work with the so-called projective description of $\S^{\{\V\}}_{\{\W\}}$. Therefore, we will first show in the next subsection that under the assumptions of Theorem \ref{t:KernelTheorems} $\S^{\{\V\}}_{\{\W\}}$ and its projective description coincide as locally convex spaces, see Theorem \ref{t:ProjDescr}. In Subsection \ref{subsect: proof kernels}, we provide a proof of Theorem \ref{t:KernelTheorems}. 
	
\subsection{The projective description of $\S^{\{\V\}}_{\{\W\}}$}
\label{sec:ProjDescr}

We consider the space $\overline{\S}^{\{\V\}}_{\{\W\}}$ of all $\varphi \in C(\R^{d})$ such that $\varphi \in \S^{v}_{w}$ for any $v \in \overline{V}(\V)$ and $w \in \overline{V}(\W)$, and we endow it with the locally convex topology generated by the system of seminorms $\{ \|\cdot\|_{\S^{v}_{w}} : v \in \overline{V}(\V), w \in \overline{V}(\W) \}$. We call the space $\overline{\S}^{\{\V\}}_{\{\W\}}$ the \emph{projective description of $\S^{\{\V\}}_{\{\W\}}$}. As sets, $\S^{\{\V\}}_{\{\W\}}$ and $\overline{\S}^{\{\V\}}_{\{\W\}}$ describe the same family of functions.

	\begin{lemma}
		\label{l:ProjectiveDescriptionContinuousEmbedding}
		Let $\V$ and $\W$ be weight function systems. Then, $\S^{\{\V\}}_{\{\W\}}$ and $\overline{\mathcal{S}}^{\{\V\}}_{\{\W\}}$ coincide algebraically and the inclusion mapping $\S^{\{\V\}}_{\{\W\}} \rightarrow \overline{\S}^{\{\V\}}_{\{\W\}}$ is continuous. 
	\end{lemma}
	
	\begin{proof}
		It is clear that the inclusion mapping $\S^{\{\V\}}_{\{\W\}} \rightarrow \overline{\S}^{\{\V\}}_{\{\W\}}$ is well-defined and continuous. Suppose now that $\varphi \in \overline{S}^{\{\V\}}_{\{\W\}}$, then $\|\varphi\|_{w} < \infty$ for any $w \in \overline{V}(\W)$ and $\|\widehat{\varphi}\|_{v} < \infty$ for any $v \in \overline{V}(\V)$. By Lemma \ref{l:BoundednessNachbinFamily} it follows that for certain $\lambda_{0}, \lambda_{1} \in \R_{+}$ we have $\|\varphi\|_{w^{\lambda_{0}}} < \infty$ and $\|\widehat{\varphi}\|_{v^{\lambda_{1}}} < \infty$. Taking $\lambda = \max\{\lambda_{0}, \lambda_{1}\}$, it follows that $\varphi \in \S^{v^{\lambda}}_{w^{\lambda}}$.
	\end{proof} 

We shall now show that the spaces $\S^{\{\V\}}_{\{\W\}}$ and $\overline{\S}^{\{\V\}}_{\{\W\}}$ also coincide topologically for two weight function systems $\V$ and $\W$ satisfying $\{\condM\}$, $\{\condSq\}$, and $\{\condN\}$. This will be done by finding the STFT characterization of the projective description and applying Lemma \ref{l:ProjDescrContinuousFunctions}. Before we are able to do this, we first need the following three lemmas on the maximal Nachbin family of a weight function system.

	\begin{lemma}
		\label{l:NachbinModerate}
		Let $\W$ be a weight function system satisfying $\{\condM\}$. Then, for any $\nu \in \R_{+}$ and $w \in \overline{V}(\W)$, there exist $\overline{w} \in \overline{V}(\W)$ and $C > 0$ such that
			\[ w(x + y) \leq C \overline{w}(x) w^{\nu}(y) , \qquad \forall x, y \in \R^{d} . \]
	\end{lemma}
	
	\begin{proof}
		Fix $\nu \in \R_{+}$. Using condition $\{\condM\}$, for any $\mu \in \R_{+}$ there are a $\lambda_{\mu} \in \R_{+}$ and $C_{\mu} > 0$ such that $w^{\lambda_{\mu}}(x + y) \leq C_{\mu} w^{\mu}(x) w^{\nu}(y)$ for all $x, y \in \R^{d}$. For any $\lambda \in \R_{+}$ we write $C'_{\lambda} = \sup_{x \in \R^{d}} w(x) / w^{\lambda}(x)$. Put $\overline{w}(x) = \inf_{\mu \in \R_{+}} C_{\mu} C'_{\lambda_{\mu}} w^{\mu}(x)$. Then, $\overline{w} \in \overline{V}(\W)$ and 
			\[ w(x + y) \leq \inf_{\mu \in \R_{+}} C'_{\lambda_{\mu}} w^{\lambda_{\mu}}(x + y) \leq \inf_{\mu \in \R_{+}} C_{\mu} C'_{\lambda_{\mu}} w^{\mu}(x) w^{\nu}(y) = \overline{w}(x) w^{\nu}(y) . \]
	\end{proof}

	\begin{lemma}
		\label{l:NachbinSq}
		Let $\W$ be a weight function system satisfying $\{\condSq\}$. Then, for any $w \in \overline{V}(\W)$, also $w^{2} \in \overline{V}(\W)$.
	\end{lemma}
	
	\begin{proof}
		By $\{\condSq\}$, for any $\lambda \in \R_{+}$ there is a $\mu \in \R_{+}$ such that $(w^{\mu}(x))^{2} \leq C w^{\lambda}(x)$ for any $x \in \R^{d}$ and some $C > 0$. Then,
			\[ \sup_{x \in \R^{d}} \frac{w(x)^{2}}{w^{\lambda}(x)} \leq C \left(\sup_{x \in \R^{d}} \frac{w(x)}{w^{\mu}(x)}\right)^{2} < \infty . \]
	\end{proof}
	
	\begin{lemma}
		\label{l:NachbinN}
		Let $\W$ be a weight function system satisfying $\{\condN\}$. Then, for any $w \in \overline{V}(\W)$, there exists a $\overline{w} \in \overline{V}(\W)$ such that $w / \overline{w} \in L^{1}(\R^{d})$.
	\end{lemma}
	
	\begin{proof}
		Using $\{\condN\}$, for any $\mu \in \R_{+}$ there exists a $\lambda_{\mu} \in \R_{+}$ such that $w^{\lambda_{\mu}} / w^{\mu} \in L^{1}(\R^{d})$ and we set $C_{\mu} = \|w^{\lambda_{\mu}} / w^{\mu}\|_{L^{1}}$. Also, for any $\lambda \in \R_{+}$ we write $C'_{\lambda} = \sup_{x \in \R^{d}} w(x) / w^{\lambda}(x)$. If we consider the function 
			\[ \overline{w}(x) = \inf_{n \geq 1} 2^{n} C_{n} C'_{\lambda_{n}} w^{n}(x) , \] 
		then $\overline{w}(x) \in \overline{V}(\W)$ and
			\begin{align*} 
				\int_{\R^{d}} \frac{w(x)}{\overline{w}(x)} dx 
					&
					\leq \int_{\R^{d}} \sup_{n \geq 1} \frac{w(x)}{2^{n} C_{n} C'_{\lambda_{n}} w^{n}(x)} dx \\
					&
					\leq \sum_{n \geq 1} \frac{1}{2^{n}} \int_{\R^{d}} (C'_{\lambda_{n}})^{-1} \frac{w(x)}{w^{\lambda_{n}}(x)} C^{-1}_{n} \frac{w^{\lambda_{n}}(x)}{w^{n}(x)} dx 
					\leq 1 .  
			\end{align*}
	\end{proof}
	
We denote by $\overline{\S}^{\{\V\}}_{\{\W\}, 1}$ the space of all $\varphi \in \S^{\prime}(\R^{d})$ such that $\|\varphi\|_{\S^{v}_{w, 1}} < \infty$ for all $v \in \overline{V}(\V)$ and $w \in \overline{V}(\W)$, and we endow it with the locally convex topology generated by the system of seminorms $\{ \|\cdot\|_{\S^{v}_{w, 1}} : v \in \overline{V}(\V), w \in \overline{V}(\W) \}$. Note that as the constant function $1$ is an element of both $\overline{V}(\V)$ and $\overline{V}(\W)$, it follows that the elements of $\overline{\S}^{\{\V\}}_{\{\W\}, 1}$ are all continuous functions. We now find the following STFT characterization.
	
	\begin{proposition}
		\label{p:STFTBBSpProjDescr}
		Let $\V$ and $\W$ be weight function systems satisfying $\{\condM\}$, $\{\condSq\}$, and $\{\condN\}$. For any $\psi \in \S^{\{\V\}}_{\{\W\}}$, the linear mappings
			\[ V_{\check{\psi}} : \overline{\S}^{\{\V\}}_{\{\W\}, 1} \rightarrow C\overline{V}(\W \otimes \V)(\R^{2d}) \quad \text{and} \quad V^{*}_{\psi} : C\overline{V}(\W \otimes \V)(\R^{2d}) \rightarrow \overline{\S}^{\{\V\}}_{\{\W\}}  \]
		are well-defined and continuous. Moreover, if $\gamma \in \S^{\{\V\}}_{\{\W\}}$ is a synthesis window for $\check{\psi}$, then,
			\begin{equation}
				\label{eq:STFTReconstructBBSpProjDescr}
				\frac{1}{(\gamma, \check{\psi})_{L^{2}}} V^{*}_{\gamma} \circ V_{\check{\psi}} = \id_{\overline{\S}^{\{\V\}}_{\{\W\}, 1}} . 
			\end{equation}
	\end{proposition}
	
	\begin{proof}
		Since for any $w \in \overline{V}(\W)$ also $\max\{1, w\} \in \overline{V}(\W)$ (similarly for $\overline{V}(\V)$) and as $\overline{V}(\W) \otimes \overline{V}(\V)$ is upward dense in $\overline{V}(\W \otimes \V)$, the continuity of the mappings follows by combining Lemmas \ref{l:STFTSpecific} and \ref{l:AdjointSTFTSpecific} with Lemmas \ref{l:NachbinModerate}, \ref{l:NachbinSq}, and \ref{l:NachbinN}, while \eqref{eq:STFTReconstructBBSpProjDescr} can be shown analogously as \eqref{eq:STFTReconstructBBSp}.  
	\end{proof}
	
We are ready to establish the main result of this subsection.

	\begin{theorem}
		\label{t:ProjDescr}
		Let $\V$ and $\W$ be weight function systems satisfying $\{\condM\}$, $\{\condSq\}$, and $\{\condN\}$. Then, $\S^{\{\V\}}_{\{\W\}}$ and $\overline{S}^{\{\V\}}_{\{\W\}}$ coincide topologically.
	\end{theorem}
	
	\begin{proof}
		By Lemma \ref{l:NachbinN} it follows that $\overline{\S}^{\{\V\}}_{\{\W\}} \subseteq \overline{\S}^{\{\V\}}_{\{\W\}, 1}$. Hence, by Lemma \ref{l:ProjectiveDescriptionContinuousEmbedding} it suffices to show that $\overline{\S}^{\{\V\}}_{\{\W\}, 1} \subseteq \S^{\{\V\}}_{\{\W\}}$ continuously. Now, if we take $\gamma, \psi \in \S^{\{\V\}}_{\{\W\}}$ such that $(\gamma, \check{\psi})_{L^{2}} = 1$ as in the proof of Proposition  \ref{p:BBSpInftyandL1}, it follows from Propositions \ref{p:STFTBBSp} and \ref{p:STFTBBSpProjDescr} and Lemma \ref{l:ProjDescrContinuousFunctions} that the following diagram commutes 
			\[ 
				\begin{tikzcd}[ampersand replacement=\&]
					\overline{\S}^{\{\V\}}_{\{\W\}, 1} \arrow{r}{V_{\check{\psi}}} \arrow{d}{\iota} \& C\overline{V}(\W \otimes \V)(\R^{2d}) = C_{\{\W \otimes \V\}}(\R^{2d}) \arrow{dl}{V^{*}_{\gamma}} \\
					\S^{\{\V\}}_{\{\W\}} \& ~ 
				\end{tikzcd} 
			\]
		which shows the claim.
	\end{proof}
	
\subsection{The proof of Theorem \ref{t:KernelTheorems}}
\label{subsect: proof kernels}
We now move on to the proof of the kernel theorem. Given two weight function systems $\V$ on $\R^{d_{1}}$ and $\W$ on $\R^{d_{2}}$, we may consider the canonical embedding $C_{[\V]}(\R^{d_{1}}) \otimes C_{[\W]}(\R^{d_{2}}) \rightarrow C_{[\V \otimes \W]}(\R^{d_{1} + d_{2}})$ given by
	\begin{equation}
		\label{eq:CanonicalEmbeddingTensProd}
		\sum_{k = 1}^{n} \Phi_{1, k} \otimes \Phi_{2, k} \longmapsto \left[ (x_{1}, x_{2}) \mapsto \sum_{k = 1}^{n} \Phi_{1, k}(x_{1}) \Phi_{2, k}(x_{2}) \right] .
	\end{equation}
Then clearly we may extend this embedding (where in the Roumieu case we use \cite[Theorem 3.7]{B-M-S-ProjDescrWeighIndLim} and that the weight function systems satisfy (S) in view of Lemma \ref{l:CondNinC0}) to the canonical isomorphism
	\[ C_{[\V]}(\R^{d_{1}}) \widehat{\otimes}_{\varepsilon} C_{[\W]}(\R^{d_{2}}) \cong C_{[\V \otimes \W]}(\R^{d_{1} + d_{2}}) . \]

We are now fully prepared to show the kernel theorem.

	\begin{proof}[Proof of Theorem \ref{t:KernelTheorems}]
		It suffices to show the isomorphisms in \eqref{eq:KernelTheoremBBSp}, as the isomorphisms in \eqref{eq:KernelTheoremDual} would then follow from those by the general theory of nuclear Fr\'{e}chet and $(DF)$-spaces, see e.g. \cite[Theorem 2.2]{K-UltraDist3}. Moreover, the canonical isomorphism $\S^{[\V_{1}]}_{[\W_{1}]}(\R^{d_{1}}) \compltens \S^{[\V_{2}]}_{[\W_{2}]}(\R^{d_{2}}) \cong \mathcal{L}_{\beta}(\S^{[\V_{1}]}_{[\W_{1}]}(\R^{d_{1}})^{\prime}, \S^{[\V_{2}]}_{[\W_{2}]}(\R^{d_{2}}))$ follows from the fact that both $\S^{[\V_{1}]}_{[\W_{1}]}(\R^{d_{1}})$ and $\S^{[\V_{2}]}_{[\W_{2}]}(\R^{d_{2}})$ are nuclear and complete, see e.g. \cite[Proposition 50.5]{T-TopVecSp}. Hence, we are left with showing $\S^{[\V_{1} \otimes \V_{2}]}_{[\W_{1} \otimes \W_{2}]}(\R^{d_{1} + d_{2}}) \cong \S^{[\V_{1}]}_{[\W_{1}]}(\R^{d_{1}}) \compltens \S^{[\V_{2}]}_{[\W_{2}]}(\R^{d_{2}})$.
	
		We consider the canonical embedding $\iota : \S^{[\V_{1}]}_{[\W_{1}]}(\R^{d_{1}}) \otimes \S^{[\V_{2}]}_{[\W_{2}]}(\R^{d_{2}}) \rightarrow \S^{[\V_{1} \otimes \V_{2}]}_{[\W_{1} \otimes \W_{2}]}(\R^{d_{1} + d_{2}})$ defined by
			\[ \sum_{k = 1}^{n} \varphi_{k} \otimes \psi_{k} \longmapsto \left[ (x_{1}, x_ {2}) \mapsto \sum_{k = 1}^{n} \varphi_{k}(x_{1}) \psi_{k}(x_{2}) \right] . \]
		We first show that $\iota$ is well-defined and moreover that $\iota : \S^{[\V_{1}]}_{[\W_{1}]}(\R^{d_{1}}) \otimes_{\pi} \S^{[\V_{2}]}_{[\W_{2}]}(\R^{d_{2}}) \rightarrow \S^{[\V_{1} \otimes \V_{2}]}_{[\W_{1} \otimes \W_{2}]}(\R^{d_{1} + d_{2}})$ is continuous. Suppose $f = \sum_{k = 1}^{n} \varphi_{k} \otimes \psi_{k}$ for certain $\varphi_{k} \in \S^{[\V_{1}]}_{[\W_{1}]}(\R^{d_{1}})$ and $\psi_{k} \in \S^{[\V_{2}]}_{[\W_{2}]}(\R^{d_{2}})$ and suppose $v_{j} = v^{\lambda}_{j}$ and $w_{j} = w^{\lambda}_{j}$ for some $\lambda \in \R_{+}$ ($v_{j} \in \overline{V}(\V)$ and $w_{j} \in \overline{V}(\W)$) for $j = 1, 2$. Then,
			\begin{align*}
				&\sup_{(x_{1}, x_{2}) \in \R^{d_{1} + d_{2}}} \left|\sum_{k = 1}^{n} \varphi_{k}(x_{1}) \psi_{k}(x_{2})\right| w_{1}(x_{1}) w_{2}(x_{2}) \\
					&\qquad \qquad + \sup_{(\xi_{1}, \xi_{2}) \in \R^{d_{1} + d_{2}}} \left|\sum_{k = 1}^{n} \widehat{\varphi}_{k}(\xi_{1}) \widehat{\psi}_{k}(\xi_{2})\right| v_{1}(\xi_{1}) v_{2}(\xi_{2}) 
					\leq 2 \sum_{k = 1}^{n} \|\varphi_{k}\|_{\S^{v_{1}}_{w_{1}}} \|\psi_{k}\|_{\S^{v_{2}}_{w_{2}}} ,
			\end{align*}
		which shows that the embedding $\iota : \S^{[\V_{1}]}_{[\W_{1}]}(\R^{d_{1}}) \otimes_{\pi} \S^{[\V_{2}]}_{[\W_{2}]}(\R^{d_{2}}) \rightarrow \S^{[\V_{1} \otimes \V_{2}]}_{[\W_{1} \otimes \W_{2}]}(\R^{d_{1} + d_{2}})$ is well-defined and continuous. Here in the Roumieu case we used Theorem \ref{t:ProjDescr}. As $\S^{[\V_{1} \otimes \V_{2}]}_{[\W_{1} \otimes \W_{2}]}(\R^{d_{1} + d_{2}})$ is complete, we may uniquely extend $\iota$ to a continuous linear mapping 
			\[ \iota : \S^{[\V_{1}]}_{[\W_{1}]}(\R^{d_{1}}) \widehat{\otimes} \S^{[\V_{2}]}_{[\W_{2}]}(\R^{d_{2}}) \rightarrow \S^{[\V_{1} \otimes \V_{2}]}_{[\W_{1} \otimes \W_{2}]}(\R^{d_{1} + d_{2}}) . \] 
		We now show that $\iota$ is a topological isomorphism. To this end, we take windows $\psi_{j}, \gamma_{j} \in \S^{[\V_{j}]}_{[\W_{j}]}(\R^{d_{j}})$, $j = 1, 2$, such that $(\gamma_{j}, \check{\psi}_{j})_{L^{2}} = 1$, and note that due to our prior observations we have that $\psi_{1} \otimes \psi_{2}$ and $\gamma_{1} \otimes \gamma_{2}$ are elements of $\S^{[\V_{1} \otimes \V_{2}]}_{[\W_{1} \otimes \W_{2}]}(\R^{d_{1} + d_{2}})$. We then consider the following diagram
			\[
				\begin{tikzcd}[column sep=huge, row sep=large, ampersand replacement=\&]
					\S^{[\V_{1}]}_{[\W_{1}]}(\R^{d_{1}}) \widehat{\otimes} \S^{[\V_{2}]}_{[\W_{2}]}(\R^{d_{2}}) \arrow{r}{V_{\check{\psi}_{1}} \widehat{\otimes}_{\varepsilon} V_{\check{\psi}_{2}}} \arrow{d}{\iota} \& C_{[\W_{1} \otimes \V_{1}]}(\R^{2d_{1}}) \widehat{\otimes}_{\varepsilon} C_{[\W_{2} \otimes \V_{2}]}(\R^{2d_{2}}) \arrow{d}{\iota^{\prime}} \arrow[l, bend right, swap, "V^{*}_{\gamma_{1}} \widehat{\otimes}_{\varepsilon} V^{*}_{\gamma_{2}}"]  \\
					\S^{[\V_{1} \otimes \V_{2}]}_{[\W_{1} \otimes \W_{2}]}(\R^{d_{1} + d_{2}}) \arrow{r}{V_{\check{\psi}_{1} \otimes \check{\psi}_{2}}} \& C_{[(\W_{1} \otimes \V_{1}) \otimes (\W_{2} \otimes \V_{2})]}(\R^{2d_{1} + 2d_{2}}) \arrow[l, bend left, "V^{*}_{\gamma_{1} \otimes \gamma_{2}}"]
				\end{tikzcd}
			\]
		where $\iota^{\prime}$ is the canonical isomorphism generated by \eqref{eq:CanonicalEmbeddingTensProd} and we observe that all mappings are continuous by Proposition \ref{p:STFTBBSp}. Consider the continuous linear mapping $A = (V^{*}_{\gamma_{1}} \widehat{\otimes}_{\varepsilon} V^{*}_{\gamma_{2}}) \circ (\iota')^{-1} \circ V_{\check{\psi}_{1} \otimes \check{\psi}_{2}}$. It is clear from \eqref{eq:STFTReconstructBBSp} that $A \circ \iota = \id$ on $\S^{[\V_{1}]}_{[\W_{1}]}(\R^{d_{1}}) \otimes \S^{[\V_{2}]}_{[\W_{2}]}(\R^{d_{2}})$, hence by its density also on $\S^{[\V_{1}]}_{[\W_{1}]}(\R^{d_{1}}) \widehat{\otimes} \S^{[\V_{2}]}_{[\W_{2}]}(\R^{d_{2}})$. Next, we claim that $\im \iota$ is dense in $\S^{[\V_{1} \otimes \V_{2}]}_{[\W_{1} \otimes \W_{2}]}(\R^{d_{1} + d_{2}})$. Indeed, this follows immediately from the density of $\iota'(C_{[\W_{1} \otimes \V_{1}]}(\R^{2d_{1}}) \otimes C_{[\W_{2} \otimes \V_{2}]}(\R^{2d_{2}}))$ in $C_{[(\W_{1} \otimes \V_{1}) \otimes (\W_{2} \otimes \V_{2})]}(\R^{2d_{1} + 2d_{2}})$, Proposition \ref{p:STFTBBSp}, and \eqref{eq:STFTReconstructBBSp}. But as $\iota \circ A = \id$ on $\im \iota$, we infer that the identity holds on the whole of $\S^{[\V_{1} \otimes \V_{2}]}_{[\W_{1} \otimes \W_{2}]}(\R^{d_{1} + d_{2}})$. Consequently, $\iota$ and $A$ are each others inverse, which implies that $\iota$ is a topological isomorphism. This completes the proof.
	\end{proof}

\appendix

\section{Completing the proof of Theorem \ref{t:NuclearityBBSp}}
\label{appendix:ProofNecessityNuclearity}

In this appendix, we show the implication $(ii) \Rightarrow (i)$ of Theorem \ref{t:NuclearityBBSp}. To do this, we proceed similarly as in \cite{D-N-V-CharNuclBBSp,D-N-V-NuclGSSpKernThm} and apply a result due to Petzsche \cite{P-NuklUltraDistSatzKern}  using K\"{o}the sequence spaces of echelon and co-echelon type \cite{B-M-S-KotheSetsKotheSeqSp, M-V-IntroFuncAnal}.

For a given sequence $a = (a_{j})_{j \in \Z^{d}}$ of positive numbers, we define the Banach space $l^{q}(\Z^{d}, a) = l^{q}(a)$, $q \in \{1, \infty\}$, of all $c = (c_{j})_{j \in \Z^{d}} \in \C^{\Z^{d}}$ such that
	\[ \|c\|_{l^{1}(a)} = \sum_{j \in \Z^{d}} |c_{j}| a_{j} < \infty , \]
and
	\[ \|c\|_{l^{\infty}(a)} =  \sup_{j \in \Z^{d}} |c_{j}| a_{j} < \infty . \]
A \emph{K\"{o}the set} is a family $A = \{ a^{\lambda} : \lambda \in \R_{+} \}$ of sequences $a^{\lambda}$ of positive numbers such that $a^{\lambda}_{j} \leq a^{\mu}_{j}$ for all $j \in \Z^{d}$ and $\mu \leq \lambda$. We define the associated \emph{K\"{o}the sequence spaces} as
	\[ \lambda^{q}(A) = \varprojlim_{\lambda \rightarrow 0^{+}} l^{q}(a^{\lambda}) , \qquad \lambda^{q}\{A\} = \varinjlim_{\lambda \rightarrow \infty} l^{q}(a^{\lambda}) , \qquad q \in \{ 1, \infty \} . \]
Then $\lambda^{q}(A)$ is a Fr\'{e}chet space, while $\lambda^{q}\{A\}$ is a regular $(LB)$-space, as follows from \cite[p.~80, Corollary 7]{B-IntroLocallyConvexIndLim}. Moreover, $\lambda^{1}[A] \subseteq \lambda^{\infty}[A]$ continuously. The nuclearity of $\lambda^{q}[A]$ can be characterized using the following conditions on the K\"{o}the set A:
	\begin{itemize}
		\item[$(\condN)$] $\forall \lambda \in \R_{+} ~ \exists \mu \in \R_{+} : a^{\lambda} / a^{\mu} \in l^{1}$;
		\item[$\{\condN\}$] $\forall \mu \in \R_{+} ~ \exists \lambda \in \R_{+} : a^{\lambda} / a^{\mu} \in l^{1}$.
	\end{itemize}
	
	\begin{lemma}[{cf. \cite[Proposition 28.16]{M-V-IntroFuncAnal} and \cite[p.~75, Proposition 15]{B-IntroLocallyConvexIndLim}}]
		\label{l:NuclKothe}
		Let $A$ be a K\"{o}the set and $q \in \{1, \infty\}$. Then, $\lambda^{q}[A]$ is nuclear if and only if $A$ satisfies $[\condN]$.
	\end{lemma}
	
Given a weight function system $\W$, we may associate to it the K\"{o}the set
	\[ A_{\W} = \{ (w^{\lambda}(j))_{j \in \Z^{d}} : \lambda \in \R_{+} \} . \]
In case $\W$ satisfies $[\condM]$, then the notion of $[\condN]$ is unambiguous, that is, by \cite[Lemma 3.2]{D-N-V-NuclGSSpKernThm}:
	\begin{equation} 
		\label{eq:WeighFuncSystemCondNEquivKotheSet}
		\W \text{ satisfies } [\condN] \quad \Longleftrightarrow \quad A_{\W} \text{ satisfies } [\condN] . 
	\end{equation}
	
We will now make use of the following result in order to complete the proof of Theorem \ref{t:NuclearityBBSp}.

	\begin{lemma}[{\cite[Lemma 5.5]{D-N-V-NuclGSSpKernThm}}]
		\label{l:PetzscheTrick}
		Let $A$ be a K\"{o}the set and $E$ be a lcHs.
			\begin{itemize}
				\item[$(a)$] Suppose that $E$ is nuclear and that there are continuous linear mappings $T : \lambda^{1}(A) \rightarrow E$ and $S : E \rightarrow \lambda^{\infty}(A)$ such that $S \circ T = \iota$, where $\iota : \lambda^{1}(A) \rightarrow \lambda^{\infty}(A)$ denotes the natural embedding. Then, $\lambda^{1}(A)$ is nuclear.
				\item[$(b)$] Suppose that $E^{\prime}_{\beta}$ is nuclear and that there are continuous linear mappings $T : \lambda^{1}\{A\} \rightarrow E$ and $S : E \rightarrow \lambda^{\infty}\{A\}$ such that $S \circ T = \iota$, where $\iota : \lambda^{1}\{A\} \rightarrow \lambda^{\infty}\{A\}$ denotes the natural embedding. Then, $\lambda^{1}\{A\}$ is nuclear.
			\end{itemize}
	\end{lemma}
	
	\begin{proof}[Proof of Theorem \ref{t:NuclearityBBSp} $(ii) \Rightarrow (i)$]
		As $\S^{[\V]}_{[\W]}$ and $\S^{[\check{\W}]}_{[\V]}$ are isomorphic via the Fourier transform, it suffices to show that $\W$ satisfies $[\condN]$. We will apply Lemma \ref{l:PetzscheTrick} with $A = A_{\W}$ and $E = \S^{[\V]}_{[\W]}$ (in the Roumieu case we use the well-known fact that the strong dual of a nuclear $(DF)$-space is nuclear). Consider for an arbitrary but fixed $\varphi_{0} \in \S^{[\V]}_{[\W]}$ the mappings
			\[ T : \lambda^{1}[A_{\W}] \rightarrow \S^{[\V]}_{[\W]} : \quad c \mapsto \sum_{j \in \Z^{d}} c_{j} \varphi_{0}(\cdot - j) , \]
		and
			\[ S : \S^{[\V]}_{[\W]} \rightarrow \lambda^{\infty}[A_{\W}] : \quad \varphi \mapsto \left( \int_{[0, \frac{1}{2}]^{d}} \varphi(x + j) dx \right)_{j \in \Z^{d}} ,  \]
		which are obviously well-defined and continuous due to condition $[\condM]$. Now note that if the $\varphi_{0} \in \S^{[\V]}_{[\W]}$ is such that
			\begin{equation}
				\label{eq:Condphi0}
				\int_{[0, \frac{1}{2}]^{d}} \varphi_{0}(x + j) dx = \delta_{j, 0} , \qquad j \in \Z^{d} ,
			\end{equation}
		then $\iota = S \circ T$, whence it would follow from Lemmas \ref{l:NuclKothe} and \ref{l:PetzscheTrick} and \eqref{eq:WeighFuncSystemCondNEquivKotheSet} that $\W$ satisfies $[\condN]$. Let us thus verify that such a $\varphi_{0}$ exists. Take any $\varphi \in \S^{[\V]}_{[\W]} \cap \S^{[\V]}_{[\W], 1}$ such that $\varphi(0) = 1$, which is possible due to our assumption of $\S^{[\V]}_{[\W]} \neq \{0\}$ and Lemma \ref{l:BBSpNonTrivialEquiv}. By possibly replacing $\varphi$ by $(\varphi * \phi_{0}) \cdot \widehat{\phi}_{1}$, where $\phi_{0}, \phi_{1} \in \D(\R^{d})$ are such that $\varphi * \phi_{0}(0) = 1$ and $\widehat{\phi}_{1}(0) = 1$, we may assume that $\partial^{\alpha} \varphi$ exists for any $\alpha \in \N^{d}$ and moreover $\partial^{\alpha} \varphi \in \S^{[\V]}_{[\W]}$. Now set
			\[ \chi(x) = 2^{-d} \int_{[-1, 1]^{d}} e^{-2 \pi i x \cdot t} dt  , \qquad x \in \R^{d} . \]
		Then, $\chi(j / 2) = \delta_{j, 0}$ for all $j \in \Z^{d}$. Hence, $\psi = \varphi \chi \in \S^{[\V]}_{[\W]}$ as well as all its derivatives, and $\psi(j / 2) = \delta_{j, 0}$ for all $j \in \Z^{d}$. Then, $\varphi_{0} = (-1)^{d} \partial^{d} \cdots \partial^{1} \psi$ is an element of $\S^{[\V]}_{[\W]}$ satisfying \eqref{eq:Condphi0}.
	\end{proof}

	\begin{remark} 
		It should be noted that the condition $[\condSq]$ is not needed for the proof above of the implication $(ii)\Rightarrow(i)$ of Theorem \ref{t:NuclearityBBSp}.
	\end{remark}

\end{document}